\documentclass[a4paper, 11pt]{amsart}

\usepackage{amsmath,amssymb,amsfonts,amsthm}

\usepackage[mathscr]{eucal}

\newtheorem{theorem}{Theorem}[section]
\newtheorem{lemma}[theorem]{Lemma}
\newtheorem{corollary}[theorem]{Corollary}

\newtheorem*{KT}{Kneser's Theorem}

\newtheorem*{DGM}{Devos-Goddyn-Mohar Theorem}

\newcommand{\N}{\mathbb N}
\newcommand{\Z}{\mathbb Z}

\newcommand{\C}{\mathcal C}

\newcommand{\B}{\mathscr B}

\newcommand{\A}{\mathscr A}
\newcommand{\Dc}{\mathscr D}
\newcommand{\Fc}{\mathcal F}
\newcommand{\vp}{\mathsf v}
\newcommand{\D}{\mathsf D}
\newcommand{\h}{\mathsf h}
\newcommand{\dd}{\mathsf d}

\newcommand{\ord}{\text{\rm ord}}

\newcommand{\supp}{\text{\rm supp}}

\newcommand{\E}{\mathsf E}
\newcommand{\und}{\;\;\;\mbox{ and }\;\;\;}

\newcommand{\be}{\begin{equation}}
\newcommand{\ee}{\end{equation}}

\newcommand{\bnml}{\begin{multline}}
\newcommand{\enml}{\end{multline}}

\newcommand{\ber}{\begin{eqnarray}}
\newcommand{\eer}{\end{eqnarray}}
\newcommand{\nn}{\nonumber}
\newcommand{\Sum}[2]{\underset{#1}{\overset{#2}{\sum}}}
\newcommand{\Summ}[1]{\underset{#1}{\sum}}
\newcommand{\fP}{\mathscr{S}}

\author{David J. Grynkiewicz$^1$}\thanks{The first author is supported by FWF project number M1014-N13}
\author{Luz Elimar Marchan$^2$}\thanks{This work was initiated while the third author was visiting the
Institut f\"ur Mathematik und Wissenschaftliches Rechnen,
Karl-Franzens-Universit\"at Graz, Austria, with partial support of the CDCH, Universidad Central de Venezuela}
\author{Oscar Ordaz$^3$}

\address{$^1$ Institut f\"ur Mathematik und Wissenschaftliches Rechnen,
Karl-Franzens-Universit\"at Graz,
Heinrichstra\ss e 36,
8010 Graz, Austria} \email{diambri@hotmail.com}

\address{$^2$ Departamento de Matem\'aticas, Decanato de Ciencias y Tecnolog\'{i}as,
Universidad Centroccidental Lisandro Alvarado, Barquisimeto,
Venezuela}

\address{$^3$ Departamento de Matem\'aticas y Centro ISYS, Facultad de Ciencias,
Universidad Central de Venezuela, Ap. 47567, Caracas 1041-A,
Venezuela}

\begin{document}
\title{A Weighted Generalization of Two Theorems of Gao}

\begin{abstract}
Let $G$ be a finite abelian group and let $A\subseteq \Z$ be nonempty. Let $\D_A(G)$ denote the minimal integer such that any sequence over $G$ of length $\D_A(G)$ must contain a nontrivial subsequence $s_1\cdots s_r$ such that $\Sum{i=1}{r}w_is_i=0$ for some $w_i\in A$. Let $\E_A(G)$ denote the minimal integer such that any sequence over $G$ of length $\E_A(G)$ must contain a subsequence of length $|G|$, $s_1\cdots s_{|G|}$, such that $\Sum{i=1}{|G|}w_is_i=0$ for some $w_i\in A$. In this paper, we show that $$\E_A(G)=|G|+\D_A(G)-1,$$ confirming a conjecture of Thangadurai and the expectations of Adhikari, et al. The case $A=\{1\}$ is an older result of Gao, and our result extends much partial work done by Adhikari, Rath, Chen, David, Urroz, Xia, Yuan, Zeng and Thangadurai. Moreover, under a suitable multiplicity restriction, we show that not only can zero be represented in this manner, but an entire nontrivial subgroup, and if this subgroup is not the full group $G$, we obtain structural information for the sequence generalizing another non-weighted result of Gao. Our full theorem is valid for more general $n$-sums with $n\geq |G|$, in addition to the case $n=|G|$.
\end{abstract}

\maketitle

\section{Introduction}
We follow the notation of \cite{GaoGer-survey} \cite{Alfred-book} \cite{Alfred-BCN-notes} \cite{oscar-weighted-projectI} concerning sumsets,  sequences and (weighted) subsequence sums. The reader less familiar with this notation can first consult the subsequent notation section, where we provide self-contained definitions for all relevant concepts.

One of the oldest questions in zero-sum additive theory (see \cite{Alfred-book} \cite{Alfred-BCN-notes} for further details), raised by Rogers et al. \cite{Rogers-Davenport} in 1962 (and later popularized by  Davenport's consideration of the same constant), is the question of, given a finite abelian group $G$, what is the minimum length so that every sequence over $G$ of this length contains a subsequence with sum zero. This number is known as the \emph{Davenport constant} of $G$, denoted $\D(G)$, and its determination is one of the more difficult open problems in the field, with its explicit value known only for a few small families of subgroups. Its value also plays an important role in controlling the behavior of factorizations over Krull Monoids.

If we modify the problem to instead ask how long must a sequence over $G$ be to guarantee a subsequence with sum zero \emph{and} length $|G|$, then we obtain the Erd\H{o}s-Ginzburg-Ziv invariant $\E(G)$. That $\E(G)\leq 2|G|-1$ was shown by Erd\H{o}s, Ginzburg and Ziv \cite{egz} in 1961, subsequently sparking a flurry of generalizations and variations. In 1995,  Gao \cite{Gao-M+D-1} \cite[Proposition 5.7.9]{Alfred-book} showed these two problems to be essentially equivalent, by establishing the identity \be\label{gao-eq}\E(G)=|G|+\D(G)-1.\ee

Around the same time, Caro conjectured \cite{carosurvery} a weighted version of the result of Erd\H{o}s, Ginzburg and Ziv. Namely, if given a sequence $W=a_1\cdots a_{|G|}$ of integers of length $|G|$ and sum zero, would any sequence over $G$ of length $2|G|-1$ contain a subsequence $s_1\cdots s_{|G|}$ with $\Sum{i=1}{|G|}a_is_i=0$. After several partial cases were confirmed  \cite{gao-wegz-partialcase} \cite{hamweightegzgroup} \cite{hamweightsrelprime},  the conjecture was finally settled recently in \cite{WEGZ}.

Since then, consideration of various other weighted zero-sum questions has received renewed interest. See \cite{zhuang-collab} \cite{oscar-weighted-projectI} \cite{ordaz-quiroz-cyliccase}  for various examples where, as in Caro's original conjecture, the length of the weight subsequence is equal to the length of the desired subsequence sum. However, the restriction of having the weight sequence be quite small in length made further progress difficult. Allowing the weight sequence $W$ to be equal to the length of the sequence $S$ over $G$, Hamidoune \cite{ham-D_A-paper} was able to give a weighted version of Gao's theorem.

But it is perhaps the weighted variation introduced by Adhikari, et al. \cite{adhi0} \cite{adhi3}
that has received the most attention in the past couple years. Here, instead of considering a fixed sequence of weights, a  nonempty subset $A\subseteq \Z$ of weights is considered. The questions then become, what is the minimum length $\D_A(G)$ so that every sequence over $G$ of this length contains a subsequence $s_1\cdots s_r$ with $\Sum{i=1}{r}a_is_i=0$ for some $a_i\in A$, and what is the minimum length $\E_A(G)$ so that every sequence over $G$ of this length contains a subsequence $s_1\cdots s_{|G|}$ of length $|G|$ with $\Sum{i=1}{|G|}a_is_i=0$ for some $a_i\in A$. Of course, this is equivalent to considering arbitrarily long sequences of weights with fixed support equal to $A$ (and no term having small multiplicity), a view we will adopt to be able to use the same notation from other weighted sequence problems.

Thus, as defined in the next section, if $S\in \Fc(G)$ is a sequence over $G$ and $W\in \Fc(\Z)$ is a sequence with support $\supp(W)=A$, then $\Sigma_n(W^n,S)$ is equal to all elements $g$ that can be represented as a sum $\Sum{i=1}{n}a_is_i=g$ with $s_1\cdots s_{n}$ a subsequence of $S$ and $a_i\in A$. Likewise, $\Sigma(W^{|S|},S)$ is equal to all elements $g$ that can be represented as a sum $\Sum{i=1}{r}a_is_i=g$ with $s_1\cdots s_{r}$ a nontrivial subsequence of $S$ and $a_i\in A$.

Many recent papers have been devoted to calculating or bounding $\D_A(G)$ and $\E_A(G)$ for various subsets $A$ and groups $G$ \cite{adhi2} \cite{adhi3} \cite{griffiths} \cite{luca} \cite{thanga-paper}. In \cite{thanga-paper}, following the common support expressed in the comments and results of several papers, a weighted form of \eqref{gao-eq} was conjectured: \be\label{conj-bound}\E_A(G)=|G|+\D_A(G)-1.\ee  Much of the work of \cite{adhi0} \cite{adhi1} \cite{adhi4} \cite{xia} \cite{weight-Gao-cyclic} was motivated by or devoted to establishing \eqref{conj-bound} under various circumstances. Of particular note is the very recent result of Yuan and Zeng \cite{weight-Gao-cyclic}, confirming \eqref{conj-bound} in the cyclic case, and the work of Adhikari and Chen \cite{adhi1}, confirming \eqref{conj-bound} when $\gcd(A-a_0)=1$, where $a_0\in A$, and which is one of the few results tackling noncyclic groups. Note the lower bound $\E_A(G)\geq |G|+\D_A(G)-1$ can be easily seen by considering an extremal sequence for $\D_A(G)$ of length $\D_A(G)-1$ that avoids representing zero using weights from $A$ concatenated with a sequence of $|G|-1$ zeros.

The main result of this paper, among other consequences, unconditionally confirms \eqref{conj-bound} for an arbitrary finite abelian group. We remark that since the case when $A=\{0\}$ is rather trivial, there is little loss of generality to assume $\gcd(A)=1$, since if $\gcd(A)=d'$, then one may simply divide each element of $A$ by $d'$, resulting in a subset of integers $A'$ with $\gcd(A')=1$, and multiply each term of the sequence $S$ by $d'$, resulting in a new sequence $S'$, and then note that the elements of $G$ which can be represented as weighted sums using the set $A$ and sequence $S$, as described above, are the same as those which can be represented as weighted sums using the set $A'$ and sequence $S'$, that is, $\Sigma_n(W^n,S)=\Sigma_n({W'}^n,S')$.

Due to subtle issues concerning translation invariance, and lack thereof, that arise from the inductive nature of the proof, we need to state Theorem \ref{thm-weighted-gao-inductive} in a rather unusual form involving an over-group $G_0$. For most applications, it likely suffices to only consider the case when $G=G_0$ with $\gamma=\delta=0$, in which case \eqref{sink} holds trivially.

\begin{theorem}\label{thm-weighted-gao-inductive} Let $G_0$ be a finite abelian group,
let $G\leq G_0$ be a subgroup, let $\gamma,\,\delta\in G_0$, let $n\geq |G|$, let
$A\subseteq \Z$ be nonempty with $\gcd(A)=1$, and let $W\in \Fc(\Z)$ be a sequence with $\supp(W)=A$ . Suppose \be\label{sink}S\in\Fc(\gamma+G)\;\;\;\mbox{ and }\;\;\; A\cdot s\subseteq \delta+G\mbox{ for all }s\in \supp(S)\ee and
$|S|\geq n+\D_A(G)-1$. Then there exist a subgroup $H\leq G$, elements $\alpha\in \gamma+G$ and $\beta\in \delta+G$,
and subsequences $S'|S$, $S''|S'$ and $S_0|S$ such that \ber
\label{setup-conds}S',\,S''&\in& \Fc(\alpha+H)\;\;\;\mbox{ and }\;\;\; A\cdot s\subseteq \beta+H\mbox{ for all }s\in \supp(S'),
\\\label{smallsubsequence}
|S'|&\geq&\min\{|S|,\,|S|-(|G/H|-2)\},
\\
|S''|&=&|H|+\D_A(H)-1\label{joyjoy},\\
\label{goal-small}|H|\beta+H&=& \Sigma_{|H|}(W^{|H|},S''),\\
\Sigma_{n}(W^{n},S)&=&(n-|H|-r)\beta+\Sigma_{|H|+r}(W^{|H|+r},S_0)\mbox{ is } H\mbox{-periodic and contains } n\beta\label{goal-big},\\\label{extra-added=thing}\hspace{.5cm}S{S'}^{-1}S''|S_0&\mbox{ and }&|S_0|=|H|+r+\D_A(H)+\D_A(G/H)-2\leq |H|+r+\D_A(G)-1,\eer where $r=|S{S'}^{-1}|\leq \max\{0,\,|G/H|-2\}$.
\end{theorem}

We remark (for $G$ nontrivial) that if \be\label{mult-hyp2}\vp_g(S)\leq |S|-|G|+1\ee for all $g\in G_0[d]$, where $\gcd((A-a_0)\cup \{m\})=d$ with $a_0\in A$, where $\vp_g(S)$ denotes the multiplicity of $g$ in $S$, and where $$G[d]=\{g\in G\mid dg=0\}$$ denotes the subgroup that is the kernel of the multiplication by $d$ homomorphism, then Theorem \ref{thm-weighted-gao-inductive} cannot hold with $H$ trivial (in view of \eqref{smallsubsequence}, \eqref{setup-conds} and Lemma \ref{coset-lemma}). Thus, under the multiplicity restriction given by \eqref{mult-hyp2} and assuming $\exp(G)|n$, we not only represent $0$ as a weighted $n$-sum, but also an entire nontrivial subgroup; in fact, the set of elements representable is periodic, and if the entire group is not representable, then we obtain structural conditions on the sequence $S$. Also worth noting are equations \eqref{goal-big} and \eqref{extra-added=thing}, which essentially show that all elements of $G$ that can be represented as weighted sums can be attained using a small subsequence of $S$.
All this mirrors similar developments and generalizations of many other non-weighted zero-sum questions (see \cite{Gao-M+D-1} \cite{WEGZ} \cite{hamconj-paper}   \cite{hamweightegzgroup} \cite{ham-ordaz}  \cite{ols8}  for several examples).

In view of the mentioned trivial lower bound construction for $\E_A(G)$ and the explanation as to why $\gcd(A)=1$ can be assumed, we see that the following is an immediate corollary of Theorem \ref{thm-weighted-gao-inductive}.

\begin{corollary}\label{thm-weight-gao-simpleversion}Let $G$ be a finite abelian group, let $S\in\Fc(G)$, let $n\geq |G|$, let $A\subseteq \Z$ be nonempty, and let $W\in \Fc(\Z)$ be a sequence with $\supp(W)=A$. If $|S|\geq n+\D_A(G)-1$, then \be\label{bab}n\cdot G\cap \Sigma_{n}(W^{n},S)\neq \emptyset.\ee In particular, $$\E_A(G)=|G|+\D_A(G)-1.$$
\end{corollary}

Finally, we mention one further result of Gao \cite{Gao-M+D-1}, generalizing an older result of Olson \cite{ols8}, which states that if $G$ is a finite abelian group, $S\in \Fc(G)$ is a sequence over $G$ of length $|G|+\D(G)-1$, and there are at most $|S|-|G/H|+1$ terms of $S$ from any coset $\alpha+H$ with $H<G$ proper, then every element of $G$ can be represented as the sum of an $|G|$-term subsequence of $S$. (In fact, it was recently shown \cite{oscar-weighted-projectI} that this result can be improved by replacing $\D(G)-1$ by $\dd^*(G)$, where $\dd^*(G)$ is as defined in the subsequent section.)

Another consequence of Theorem \ref{thm-weighted-gao-inductive} is the following immediate corollary, generalizing  this second result of Gao.

\begin{corollary}\label{thm-weight-gao-simpleversionii}Let $G$ be a finite abelian group, let $S\in\Fc(G)$, let $n\geq |G|$, let $A\subseteq \Z$ be nonempty with $\gcd(A)=1$, and let $W\in \Fc(\Z)$ be a sequence with $\supp(W)=A$. Suppose $|S|\geq n+\D_A(G)-1.$ Then either $$ \Sigma_{n}(W^{n},S)=G,$$ or else there are $\alpha,\,\beta\in G$, $H<G$ and $S'|S$ such that \ber\nn|S'|&\geq& |S|-|G/H|+2,\\ \supp(S')\subseteq \alpha+H&\und&A\cdot s\subseteq \beta+H\mbox{ for all }\nn
s\in \supp(S').\eer
\end{corollary}

\bigskip
\section{Notation and Preliminaries} \label{1}
\bigskip

\subsection{Sumsets} Throughout, all abelian groups will be written additively.
Let $G$  be an abelian group, and let  $A,\, B \subseteq G$ be nonempty subsets. Then  $$A+B = \{a+b
\mid a \in A,\, b \in B \}$$  denotes their  \emph{sumset}. We abbreviate $A\underset{h}{\underbrace{+\ldots +}}A$ by $hA$. For $g\in G$, we let $g+A=\{g+a\mid a\in A\}$ and $-A=\{-a\mid a\in A\}$. For $k\in \Z$ and $A\subseteq G$, we let $$k\cdot A=\{ka\mid a\in A\},$$ and for $A\subseteq \Z$ and $g\in G$, we let $$A\cdot g=\{ag\mid a\in A\}.$$ The  \emph{stabilizer}  of $A$ is defined as $$\mathsf{H}(A) = \{ g \in G \mid g +A = A\},$$ and $A$ is called \emph{periodic}  if $\mathsf{H}(A) \ne \{0\}$, and \emph{aperiodic} otherwise. If $A$ is a union of $H$-cosets (i.e., $H\leq \mathsf{H}(A)$), then we say $A$ is \emph{$H$-periodic}. The order of an element $g\in G$ is denoted $\ord(g)$, and we use $$\phi_H:G\rightarrow G/H$$ to denote the natural homomorphism modulo $H$. For $A\subseteq \Z$, we use $\gcd(A)$ to denote the greatest common divisor of the elements of $A$.

\subsection{Sequences}
Given a set $G_0$ (often a subset of an abelian group), we let $\Fc(G_0)$ denote the free abelian monoid with basis $G_0$ written multiplicatively. The elements of $\Fc(G_0)$ are then just multi-sets over $G_0$, but following long standing tradition, we refer to the $S\in \Fc(G_0)$ as \emph{sequences}. We write
sequences $S \in \mathcal F (G_0)$ in the form
$$
S =  s_1\cdots s_r=\prod_{g \in G_0} g^{\vp_g (S)}\,, \quad \text{where} \quad
\vp_g (S)\geq 0\mbox{ and } s_i\in G_0.
$$
We call $|S|:=r=\Summ{g\in G_0}\vp_g(S)$ the \emph{length} of $S$, and  $\mathsf v_g (S)\in \N_0$  the \ {\it multiplicity} \ of $g$ in
$S$. The \emph{support} of $S$ is $$\supp(S):=\{g\in G_0\mid \vp_g(S)>0\}.$$  A sequence $S_1 $ is called a \emph{subsequence}  of
$S$ if $S_1 | S$  in $\mathcal F (G_0)$  (equivalently,
$\vp_g (S_1) \leq\vp_g (S)$  for all $g \in G_0$), and in such case, $S{S_1}^{-1}$ or ${S_1}^{-1}S$ denotes the subsequence of $S$ obtained by removing all terms from $S_1$.
Given two sequences $S,\,T\in
\Fc(G_0)$, we use $\gcd(S,T)$ to denote the longest subsequence dividing
both $S$ and $T$, and we let $$\h(S):=\max \{\vp_g (S) \mid g \in G_0 \}$$ denote the maximum multiplicity of a term of $S$. If $\h(S)=1$, then we say $S$ is \emph{squarefree}.
Given any map $\varphi: G_0\rightarrow G'_0$, we extend $\varphi$ to a map of sequences,
$\varphi: \Fc(G_0)\rightarrow \Fc(G'_0)$,  by letting $\varphi(S):=\varphi(s_1)\cdots \varphi(s_r)$.

\subsection{Setpartitions} For  a subset $G_0$ of an abelian group, let $\mathscr{S}(G_0)=\Fc(X)$, where $X$ is the set of all finite, nonempty subsets of $G_0$. A \emph{setpartition} over $G_0$ is a sequence $\A=A_1\cdots A_n\in \mathscr{S}(G_0)$, where $A_i\subseteq G_0$ are finite and nonempty. When we refer to an \emph{$n$-setpartition}, we mean a setpartition of length $n$. The sequence partitioned by $\A$ is then $$\mathsf{S}(\A):=\prod_{i=1}^n\prod_{g\in A_i}g\in \Fc(G_0),$$ and a sequence $S\in \Fc(G_0)$ is said to have an $n$-setpartition if $S=\mathsf{S}(\A)$ for some $\A\in \fP(G_0)$ with $|\A|=n$; necessary and sufficient conditions for this to hold (see \cite[Proposition 2.1]{EGZ-II}) are that $$\h(S)\leq n\leq |S|.$$

\subsection{Subsequence Sums} If $G_0$ is a subset of an abelian group $G$ and $S=s_1\cdots s_r\in \Fc(G_0)$, with $s_i\in G_0$, then the \emph{sum} of $S$ is $$\sigma(S):=\Sum{i=1}{r}s_i=\Summ{g\in G_0}\vp_g(S)g.$$ We say $S$ is \emph{zero-sum}  if $\sigma(S)=0$. We adapt the convention that the sum of the trivial/empty sequence is zero. We say that at most $n$ terms of the
sequence $S$ are from a given subset $A \subseteq G$ if
$$
|\{ i \in [1,r] \mid s_i \in A \} | \le n.$$ Similar language is used when counting the number of terms of $S$ with (or without) a certain property (always counting terms with multiplicity).
For $g\in G$ and $w\in \Z$, we let $g+S=(g+s_1)\cdots(g+s_r)\in\Fc(g+G_0)$ and $w*S=(ws_1)\cdots (ws_r)\in \Fc(w\cdot G_0)$.

Let $S\in\Fc(G_0)$, $W\in \Fc(\Z)$ and $s=\min\{|S|,\,|W|\}$.
Define $$W \cdot S = \bigl\{\sum_{i=1}^s w_i g_i \mid  w_1 \cdots
 w_s \ \text{is a subsequence of } \ W \ \text{and} \
  g_1 \cdots  g_s \ \text{is a subsequence of } \ S
 \bigr\},$$ and for $1\leq n\leq s$, let
\ber\Sigma_n(W,S) &=& \left\{g\in W'\cdot S': S'|S,\,W'|W\mbox{ and }|W'|=|S'|=n\right\}\nn\\\nn
\Sigma_{\leq
n}(W,S)&=&\bigcup_{i=1}^{n}\Sigma_i(W,S)\quad \mbox{ and }\quad \Sigma_{\geq n}(W,S)=\bigcup_{i=n}^{s}\Sigma_i(W,S),\\\nn\Sigma(W,S)&=&\Sigma_{\leq s}(W,S).\nn\eer If $W=1^{|S|}$, then $\Sigma(W,S)$ (and other such notation) is abbreviated by $\Sigma(S)$, which is the usual notation for the set of \emph{subsequence sums}. Note that $\Sigma_{|W|}(W,S)=W\cdot S$ when $|W|\leq |S|$, and that the $\Sigma$-notation can likewise be extended to define $\Sigma_n(\A)$ with $\A\in \fP(G_0)$, etc. For example, 
$$\Sigma_n(\A)=\{\Sum{i=1}{n}A_i\mid A_1\cdots A_n\mbox{ is a sub-setpartition of }\A\}.$$

If $\A\in \fP(G_0)$ is a setpartition, we let $$\Sigma^\cup(\A)=\bigcup_{A\in \Sigma(\A)}A\und \Sigma_n^\cup(\A)=\bigcup_{A\in \Sigma_n(\A)}A$$ and likewise extend other similar $\Sigma$ notation for sequences (if one associates a sequence with the corresponding setpartition having all sets of size one, then the superscript $\cup$ would no longer be necessary). Also, $$g+\A:=(g+A_1)\cdots(g+A_r)\in  \fP(g+G_0).$$

\subsection{Preliminary Results}
For the proof, we will need several basic facts concerning the Davenport constant. First, if $G\cong C_{n_1}\oplus \ldots\oplus C_{n_r}$ with $n_1|\ldots|n_r$, then define $$\dd^*(G):=\Sum{i=1}{r}(n_i-1).$$ Second, \be\label{weighted-triv-dav-bound}\D_A(G)\leq \D(G)\ee for all nonempty $A\subseteq \Z$, as is easily seen by taking any $S\in \Fc(G)$ with $|S|=\D(G)$ and applying the definition of $\D(G)$ the sequence $a*S$, where $a\in A$. Next
(see \cite[Propositions 5.1.4 and 5.1.8]{Alfred-book}), \be\label{trivial-dav-bound}\dd^*(G)<\D(G)\leq |G|.\ee

For the proof, we will use \emph{both} the Devos-Goddyn-Mohar Theorem \cite{Devos-setpartition} and a recently proved consequence \cite{oscar-weighted-projectI} of the partition theorem from \cite{ccd}.

\begin{DGM}\label{DGM-devos-etal-thm} Let $G$ be an abelian group, let $\A\in \fP(G)$ be a setpartition, and let $n\in \Z^+$ with $n\leq |\A|$. If $H=\mathsf{H}(\Sigma^\cup_n(\A))$, then $$|\Sigma^\cup_n(\mathscr{A})|\geq (\Summ{g\in G/H}\min\{n,\,\vp_g(\mathsf{S}(\phi_H(\A))\}-n+1)|H|.$$
\end{DGM}

We remind the reader that necessary and sufficient conditions for a sequence $S\in \Fc(G)$ to have an $n$-setpartition are $\h(S)\leq n\leq |S|$.

\begin{theorem}\label{CCD-d*} Let $G$ be a finite abelian group, let $S,\,S'\in\Fc(G)$ with $S'|S$, and let $n\geq \dd^*(G)$ with $\h(S')\leq n\leq |S'|$. Then $S$ has a subsequence $S''$ with $|S''|=|S'|$ such that there is an $n$-setpartition $\A=A_1\cdots A_n\in \fP(G)$ with $\mathsf{S}(\A)=S''$ and either:

(i)  $|\Sum{i=1}{n}A_i|\geq \min\{|G|,\,|S'|-n+1\},$ or

(ii) there exist a proper, nontrivial subgroup $H$ and $\alpha\in G$ such that $\Sum{i=1}{n}A_i$ is $H$-periodic, $|\Sum{i=1}{n}A_i|\geq (e+1)|H|$, and all but $e\leq |G/H|-2$ terms of $S$ are from the same coset $\alpha+H$.
\end{theorem}

We will also need Kneser's Theorem \cite{kt} \cite{natbook} \cite{taobook}, which is a particular case of the Devos-Goddyn-Mohar Theorem \cite{Devos-setpartition}.

\begin{KT} Let $G$ be an abelian group and let $A_1,\ldots,A_n\subseteq G$. If $H=\mathsf{H}(\Sum{i=1}{n}A_i)$, then $$|\Sum{i=1}{n}\phi_H(A_i)|\geq
\Sum{i=1}{n}|\phi_H(A_i)|-n+1.$$
\end{KT}

\bigskip
\section{The Proof}
\bigskip

We begin with a series of simple lemmas.

\bigskip

\begin{lemma}\label{D-lemma} Let $G$ be an abelian group, let $H\leq G$ and let $A\subseteq \Z$ be nonempty. Then $$\D_A(G)\geq \D_A(H)+\D_A(G/H)-1.$$
\end{lemma}

\begin{proof} If $\D_A(H)$ or $\D_A(G/H)$ is infinite, then clearly so is $\D_A(G)$. Therefore we may assume both these quantities are finite. Let $S\in \Fc(H)$ be a subsequence of length $\D_A(H)-1$ such that $0\notin \Sigma(W^{|S|},S)$, where $W\in \Fc(\Z)$ is a subsequence with $\supp(W)=A$, and let $T\in \Fc(G)$ be a subsequence of length $\D_A(G/H)-1$ such that $0\notin \Sigma(W^{|T|},\phi_H(T))$. Then $0\notin \Sigma(W^{|ST|},ST)$, whence $\D_A(G)-1\geq |S|+|T|= \D_A(H)+\D_A(G/H)-2$.
\end{proof}

\bigskip

\begin{lemma}\label{translation-lemma} Let $G$ be an abelian group of exponent $m$, let $S\in \Fc(G)$, let $n\in \Z^+$ with $n\leq |S|$, let $A\subseteq \Z$ be nonempty with $\gcd((A-a_0)\cup \{m\})=d$, where $a_0\in A$, and let $W\in \Fc(\Z)$ be a sequence with $\supp(W)=A$. Then $$\Sigma_n(W^n,g+S)=na_0g+\Sigma_n(W^n,S)$$ for every $g\in G[d]$.
\end{lemma}

\begin{proof} Note $\gcd((A-a_0)\cup \{m\})=d$ implies $A\subseteq a_0+d\Z$ with $d|m$.
Now, by our hypotheses, we have $$\Sum{i=1}{n}w_{j_i}(g+s_i)=\Sum{i=1}{n}a_0 g+\Sum{i=1}{n}w_{j_i}s_i=na_0g+\Sum{i=1}{n}w_{j_i}s_i,$$ for $g\in G[d]$ and any $s_1\cdots s_n|S$ with $w_{j_i}\in A$ and $s_i\in G$.
\end{proof}

\bigskip

\begin{lemma}\label{coset-lemma} Let $G$ be an abelian group of exponent $m$, let $K\leq G$ and $\beta\in G$, let $A\subseteq \Z$ be nonempty with $\gcd(A\cup\{m\})=1$ and $\gcd((A-a_0)\cup \{m\})=d$, where $a_0\in A$. Let $g_1,\,g_2\in G$ and suppose $A\cdot g_i\subseteq \beta+K$ for $i=1,2$. Then $dg_1,\,dg_2\in K$ and there exists $\alpha\in G$ such that $g_1,\,g_2\in \alpha+K$.
\end{lemma}

\begin{proof}
Let $A=\{a_0,a_1,\ldots,a_r\}$. Then $A\cdot g_i\subseteq \beta+K$ implies $(a_j-a_0)g_i\in K$ for $j=0,1,\ldots,r$. Since $\gcd((A-a_0)\cup \{m\})=d$, we can find an appropriate linear combination of the $(a_j-a_0)g_i$ which is equal to $dg_i$, whence $dg_i\in K$, as desired.

Since $A\cdot g_i\subseteq \beta+K$ for $i=1,2$, we have $a_j(g_1-g_2)\in K$ for all $j$. Thus, since $\gcd(A\cup\{m\})=1$, we can find an appropriate linear combination of the $a_j(g_1-g_2)$ which is equal to $g_1-g_2$, whence $g_1-g_2\in K$, as desired.
\end{proof}

\bigskip

\begin{lemma}\label{confusing-lemma} Let $G_0$ be a finite abelian group, let $H\leq G\leq G_0$ be subgroups, let $\alpha\in G_0$,  let $A\subseteq \Z$ be nonempty, and let $S\in\Fc(\alpha+G)$ with $|S|\geq \D_A(G/H)$. If $A\cdot \alpha\subseteq \beta+H$ for some  $\beta \in G_0$, then there exists a nontrivial subsequence $s_1\cdots s_r|S$ and $w_{j_i}\in A$ such that $\Sum{i=1}{r}w_{j_i}s_i\in r\beta+H$, where $s_i\in \alpha+G$.
\end{lemma}

\begin{proof}
Applying the definition of $\D_A(G/H)$ to the sequence $\phi_H(-\alpha+S)\in\Fc(G/H)$, we find a nontrivial subsequence $s_1\cdots s_r|S$ and $w_{j_i}\in A$ such that $\Sum{i=1}{r}w_{j_i}(-\alpha+s_i)\in H$, where $s_i\in \alpha+G$. Since $A\cdot \alpha\subseteq \beta+H$, we have $w_{j_i}\alpha\in \beta+H$ for all $i$, and the result follows.
\end{proof}

\bigskip

\begin{lemma}\label{keylemma} Let $G$ be a finite abelian group, let $K\leq G$ be a nontrivial subgroup, and let $\A=A_1\cdots A_r\in \fP(G)$ be a setpartition with $|A_i|\geq 2$ and $|\phi_K(A_i)|=1$ for $i=1,\ldots,r$. If $r\geq |K|-1$, then there exists a nontrivial subgroup $H\leq K$ and sub-setpartitions $\A'|\A$ and $\A''|\A'$ such that  \be\label{tryst}|\A''|=|H|-1,\;\;\;\; |\A'|\geq \min\{|\A|,\, |\A|-|K/H|+2\},\;\;\;|\sigma(\A'')|=|H|\;\;\;\mbox{ and }\;\;\;|\phi_H(A_j)|=1,\ee for all $A_j\in \supp(\A')$.
\end{lemma}

\begin{proof} Consider a counterexample $\A$ with $|K|$ minimal.
If $|\sigma(\A)|=|\Sum{i=1}{r}A_i|=|K|$, then a simple greedy algorithm shows that the lemma holds with $H=K$ (see \cite[Proposition 2.2]{EGZ-II}). Otherwise, Kneser's Theorem implies that $H=\mathsf{H}(\sigma(\A))$ is a proper and nontrivial subgroup of $K$ with at most $|K/H|-2$ sets $A_i|\A$ having $|\phi_H(A_i)|>1$. Let $\A_0|\A$ be the sub-setpartition consisting of all sets $A_j|\A$ with $|\phi_H(A_j)|=1$. Hence $$|\A_0|\geq |\A|-|K/H|+2\geq |K|-|K/H|+1\geq |H|.$$ Since $H<K$, by the minimality of $\A$ we can apply the lemma to $\A_0$. Let $\A'|\A_0$ and $\A''|\A_0$ be the resulting setpartitions and let $H'\leq H$ be the resulting subgroup. However, noting that $$|\A{\A'}^{-1}|\leq |K/H|-2+|H/H'|-1\leq |K/H'|-2,$$ we see that the result holds for $\A$ using $H'$, $\A'$ and $\A''$.
\end{proof}

\bigskip

We can now proceed with the proof of Theorem \ref{thm-weighted-gao-inductive}.

\bigskip

\begin{proof}[Proof of Theorem \ref{thm-weighted-gao-inductive}] Let $$\gcd((A-a_0)\cup \{m\})=d,$$ where $a_0\in A$ and $m$ is the exponent of $G_0$.
Observe there is no loss of generality to assume \be\label{gen-cond}\langle \gamma+G\rangle =G_0\;\;\;\mbox{ and }\;\;\; A\subseteq \{1,2,\ldots,m\}.\ee
For $i=1,2,\ldots,|S|$, let $A_i=A\cdot s_i$, where $S=s_1\cdots s_{|S|}$ with $s_i\in \gamma+G$. Let $\A=A_1\cdots A_{|S|}$ be the setpartition given by the $A_i$ and observe that $$\Sigma_{n}(W^{n},S)=\Sigma^{\cup}_n(\A)\subseteq n\delta+G.$$ Throughout the proof, we will often switch between the equivalent formulations $\Sigma^{\cup}_n(\A)$ and $\Sigma_{n}(W^{n},S)$, for various sequences, often inductively applying Theorem \ref{thm-weighted-gao-inductive} to some setpartition $\mathscr{D}=D_1\cdots D_r$ when each $D_i$ is of the required form $D_i=A\cdot g_i$, for some group element $g_i$, without explicitly defining the implicit sequence $g_1\cdots g_r$.

\bigskip

\textbf{Remark A.} Applying Lemma \ref{coset-lemma} with $K$ trivial shows that $|A\cdot g|=1$ if and only if $g\in G_0[d]$, and that if  $A\cdot g=\{g'\}$, then there is a unique $g\in G_0[d]$ such that $A\cdot g=\{g'\}$.

\bigskip

\textbf{Remark B.} In view of Lemma \ref{coset-lemma}, observe, for $S'|S$ and $H\leq G$, that there existing $\beta\in G_0$ such that $A\cdot s\subseteq \beta+H$ for all $s\in \supp(S')$ is equivalent to there existing $\alpha\in G_0$ such that $S'\in \Fc(\alpha+H)$ and $A\cdot \alpha\subseteq \beta+H$. Indeed, in view of Lemma \ref{coset-lemma}, $\alpha$ is uniquely determined mod $H$ for $\beta$, and vice versa, and we have   \be\label{lucc}\phi_H(\beta)=\phi_H(a_0\alpha).\ee  Moreover, in view of \eqref{sink}, such $\alpha$ and $\beta$ can always be chosen so that $\alpha\in \gamma+G$ and $\beta\in \delta+G$.

\bigskip

\textbf{Step 1.} First, let us show that it suffices to show there exists a subgroup $H\leq G$, $\alpha\in \gamma+G$, $\beta\in \delta+G$, and subsequence $S''|S'$, where $S'|S$ is the subsequence of all terms from $\alpha+H$, such that \eqref{setup-conds}, \eqref{joyjoy} and \eqref{goal-small} hold with \be\label{weak-small-sub} |S'|\geq |H|+\D_A(G)-1.\ee Note that \eqref{smallsubsequence} and $|S|\geq |G|+\D_A(G)-1$ imply \eqref{weak-small-sub}, so that the assumption \eqref{weak-small-sub} is weaker than \eqref{smallsubsequence}. Assume $H$ is maximal such that all these assumptions hold.
By re-indexing the $A_i$, we have (in view of \eqref{setup-conds} and \eqref{goal-small})
\ber\label{trixt1} A_i&\subseteq& \beta+H,\;\;\;\mbox{ for } i=1,2,\ldots,|S'|,\\|H|\beta+H &=& \Sigma^\cup_{|H|}(A_1\cdots A_{|S''|}).\label{trixt2}\eer Let $\A'=A_1\cdots A_{|S'|}$ and $\A''=A_1\cdots A_{|S''|}$.

\bigskip

\textbf{Sub-Step 1.1.} We begin the step by proceeding to show \be\label{magicstuff}-n\beta+\Sigma^\cup_n(\A)=\Sigma^\cup_n(-\beta+\A)=\Sigma^\cup(-\beta+\A)\;\mbox{ is }H\mbox{-periodic and contains }0.\ee
Since the first equality in \eqref{magicstuff} is trivial, we need only show $\Sigma^\cup_n(\B)=\Sigma^\cup(\B)$ is $H$-periodic and contains 0, where $\B=-\beta+\A$.
Let $\phi_H(g)\in \{0\}\cup\Sigma^\cup(\phi_H(\B))$, where $g\in G$, be arbitrary. We proceed to show
$$g+H\subseteq \Sigma^\cup_{n}(\B),$$ which will establish that $\Sigma^\cup_n(\B)=\Sigma^\cup(\B)$ is $H$-periodic and contains zero, and thus show \eqref{magicstuff}.

Since $\phi_H(g)\in\{0\}\cup \Sigma^\cup(\phi_H(\B))$, we can find a sub-setpartition (possibly empty) $\C|\B$ such that $\phi_H(g)\in \sigma(\phi_H(\C))$ and $\{0\}\notin \supp(\phi_H(\C))$ (where we adopt the convention that the sum of the empty setpartition is $\{0\}$), and hence $\C|{\B'}^{-1}\B$, where $\B'=-\beta+\A'$ (in view of \eqref{trixt1}). Observing from  \eqref{trivial-dav-bound} that $$n-|H|\geq |G|-|H|\geq |G/H|-1\geq \D(G/H)-1,$$ we see that, by repeated application of the definition of $\D(G/H)$, we may assume $|\C|\leq n-|H|$.
Note $\gamma+G=\alpha+G$ (recall $\alpha\in \gamma+G$) and that Remark B and \eqref{setup-conds} imply $A\cdot \alpha\subseteq \beta+H$. Thus, considering any sub-setpartition $\A'_0|\A'$ with $\A''|\A_0'$ and $|\A'_0|=|H|+\D_A(G)-1$ (possible in view of \eqref{joyjoy} and \eqref{weak-small-sub}), we can, by repeated application of Lemma \ref{confusing-lemma} to ${(\beta+\C)}^{-1}{\A_0'}^{-1}\A$, extend $\C$ to a setpartition $\C|(-\beta+\A_0')^{-1}\B$ with $\phi_H(g)\in \sigma(\C)$ and $$n-|H|-(\D_A(G/H)-1)\leq |\C|\leq n-|H|.$$ Then, since \eqref{joyjoy} and Lemma \ref{D-lemma} imply \be\label{gummy}|{\A''}^{-1}\A'_0|=\D_A(G)-\D_A(H)\geq \D_A(G/H)-1,\ee it follows that we can append on the appropriate number of sets from $-\beta+\A''^{-1}\A_0'$ to obtain a sub-setpartition $\C'|{(-\beta+\A'')}^{-1}\B$ of length $n-|H|$ with $\phi_H(g)\in \sigma(\phi_H(\C'))$. But now $$g+H\subseteq \Sigma^\cup_{n}(\B)$$ is clear in view of \eqref{trixt2} (recall $\A''|\A'_0$ and $\B=-\beta+\A$). Thus \eqref{magicstuff} is established, completing the sub-step.

\bigskip

\textbf{Sub-Step 1.2.} Next, we show that \eqref{smallsubsequence} holds.
Factor $\B=(-\beta+\A')\Dc_0$, let $\beta+\Dc_0=A_{j_1}\cdots A_{j_r}$, and let $\Dc=\prod_{i=1}^r(\{0\}\cup (-\beta+A_{j_i}))\in \fP(G)$ (recall $\beta\in \delta+G$, so $\beta+G=\delta+G$).

Suppose instead that \eqref{smallsubsequence} fails. Then, since \eqref{smallsubsequence} is trivial when $H=G$, we may assume
$H<G$ is proper, $$|\Dc|\geq |G/H|-1,$$ and every set dividing $\Dc$ contains two distinct elements modulo $H$ (in view of Remark B and the definitions of $\Dc$, $\A'$ and $S'$).  Consequently, we can apply Lemma \ref{keylemma} to $\phi_H(\Dc)$ (with $K$ and $G$ both taken to be $G/H$ in the lemma). Let $K/H\leq G/H$ be the resulting nontrivial subgroup and let $\Dc'|\Dc$ and $\Dc''|\Dc'$ be the resulting sub-setpartitions, say w.l.o.g. $$\Dc''=\prod_{i=1}^{|K/H|-1}(\{0\}\cup (-\beta+A_{j_i}))$$ and $\Dc'=\prod_{i=1}^{|\Dc'|}(\{0\}\cup (-\beta+A_{j_i}))$. Thus  $-\beta+A_{j_i}\subseteq K$ for $i\leq |K/H|-1$ and $\sigma(\phi_H(\Dc''))=K/H$; moreover, at least  \be\label{chowder}\min\{|\A|,\,|\A|-|G/K|+2\}\ee sets $A_j|\A$ have $-\beta+A_j\subseteq K$.

In view of \eqref{joyjoy}, \eqref{weak-small-sub}  and Lemma \ref{D-lemma}, let $\A'''|\A'{\A''}^{-1}$ be a sub-setpartition of length $\D_A(K/H)-1$, and let $\C=(-\beta+\A''')\prod_{i=1}^{|K/H|-1}(-\beta+A_{j_i})$. Observe that we have $A_j\subseteq \beta+K$ for every $A\cdot s_j=A_j|(\beta+\C)$. Thus, by Remark B, there is  $\alpha'\in G_0$ such that $s_j\in \alpha'+K$ for all $A\cdot s_j=A_j|(\beta+\C)$. Moreover, $$\phi_K(\beta)=\phi_K(a_0\alpha')$$ and $\phi_K(\alpha')$ is unique subject to $A\cdot \alpha'\subseteq \beta+K$. However, in view of \eqref{setup-conds} and Remark B, we also have $\phi_H(\beta)=\phi_H(a_0\alpha)$ with $A\cdot \alpha\subseteq \beta+H\subseteq \beta+K$. Consequently, $\phi_K(\alpha')=\phi_K(\alpha)$, so that w.l.o.g. we can assume $\alpha'=\alpha$.

Observe, since $\phi_H(-\beta+A_j)=\{0\}$ for $A_j|\A'''$, and in view of the previous paragraph, that \be\label{trueiy}\phi_H\left(\Sigma^\cup_{|K/H|-1}
\left((-\beta+\A''')\prod_{i=1}^{|K/H|-1}(-\beta+A_{j_i})\right)\right)=\sigma(\phi_H(\Dc''))=K/H\ee by the same extension arguments (via Lemma \ref{confusing-lemma}) used to establish \eqref{magicstuff}.

Extend $\A''\A'''\prod_{i=1}^{|K/H|-1}A_{j_i}|\A$ to a sub-setpartition $\C'$ of length $|K|+\D_A(K)-1$ with all sets dividing $\C'$ being contained in $\beta+K$; possible (in view of \eqref{chowder}), since $$|\A|-(|G/K|-1)\geq |G|+\D_A(G)-|G/K|\geq |K|+\D_A(K)-1$$ and since (in view of Lemma \ref{D-lemma}) \ber\nn
|\A''\A'''\prod_{i=1}^{|K/H|-1}A_{j_i}|=
\nn|H|+\D_A(H)-1+\D_A(K/H)-1+(|K/H|-1)\leq |K|+\D_A(K)-1.\eer  In view of \eqref{trueiy} and \eqref{goal-small}, it follows that $$\Sigma^\cup_{|H|+|K/H|-1}(-\beta+\A''\A'''\prod_{i=1}^{|K/H|-1}A_{j_i})=K.$$ Thus, in view of \ber\nn &&|H|+|K/H|-1\leq |K|,\\\nn
&&|\A''\A'''\prod_{i=1}^{|K/H|-1}A_{j_i}|+(|K|-(|H|+|K/H|-1))=
|H|+\D_A(H)-1+\D_A(K/H)-1\\&&+(|K/H|-1)+(|K|-|H|-|K/H|+1)\leq |K|+\D_A(K)-1,\nn\eer where the last inequality is a consequence of Lemma \ref{D-lemma}, it follows that $$\Sigma_{|K|}^\cup(-\beta+\C')=K.$$ Hence, from Remark B and \eqref{chowder}, it follows that the hypotheses assumed at the beginning of Step 1 hold with subgroup $K$, contradicting the maximality of $H$. So we may assume \eqref{smallsubsequence} holds, completing the sub-step.

\bigskip

\textbf{Sub-Step 1.3.} Finally, we conclude Step 1 by showing that \eqref{goal-big} and \eqref{extra-added=thing} hold. Since \eqref{goal-big} and \eqref{extra-added=thing} follow trivially from \eqref{smallsubsequence}, \eqref{joyjoy}, \eqref{goal-small}  and \eqref{magicstuff} when $H=G$ (taking $S_0=S''$), we may assume $H<G$ is a proper subgroup. From \eqref{magicstuff}, we have \be\label{toytoy}\Sigma^\cup_n(\B)=\Sigma^\cup(\B)=\sigma(\Dc)+H,\ee where $\Dc$ and $\Dc_0$ are as defined at the beginning of Sub-Step 1.2. Note $|\Dc_0|=|\Dc|=r$.

The remainder of the argument is a simplification of the end of Sub-Step 1.2. We take a subsequence $\A'''|\A'{\A''}^{-1}$ of length $\D_A(G/H)-1$ (possible in view of \eqref{weak-small-sub}, \eqref{joyjoy} and Lemma \ref{D-lemma}) and observe that $$\phi_H\left(\Sigma^\cup_{r}\left((-\beta+\A''')\Dc_0\right)\right)=\sigma(\phi_H(\Dc))$$ by the same extension arguments (via Lemma \ref{confusing-lemma}) used to establish \eqref{magicstuff}. As a result, taking $S_0|S$ to be the subsequence corresponding to $\A''\A'''(\beta+\Dc_0)$, we see that \eqref{goal-big} and \eqref{extra-added=thing} hold in view of \eqref{toytoy}, \eqref{magicstuff}, \eqref{goal-small}, \eqref{joyjoy} and Lemma \ref{D-lemma}, completing Step 1.

\bigskip

In view of Step 1, it now suffices to prove the theorem in the case $n=|G|$ (simply apply the case $n=|G|$ to $S$ to obtain the hypotheses of Step 1, and then apply Step 1), and so we henceforth assume $n=|G|$. We may also assume \be\label{mult-hyp}\vp_g(S)\leq \D_A(G)-1\ee for all $g\in G_0[d]$, since otherwise the hypotheses of Step 1 hold with $H$ trivial, whence Step 1 again completes the proof.

\bigskip

\textbf{Step 2.} Next, we complete the case when $|A|=1$.
We proceed by induction on $|G|$.  Since $|A|=1$, we see that the hypothesis $\gcd(A)=1$ implies w.l.o.g. $A= \{1\}$. Thus \be\label{horse}\D_A(G)=\D(G).\ee Note $|A|=1$ implies $d=m$, whence \be\label{grasss}G_0[d]=G_0[m]=G_0.\ee Hence, in view of Lemma \ref{translation-lemma} and Step 1, we may w.l.o.g. assume $\vp_0(S)=\h(S)$ and $G_0=G$.

In view of \eqref{mult-hyp}, \eqref{grasss} and \eqref{horse}, it follows that there exists a subsequence $S'|S$ with $|S'|=|G|+\D(G)-1$ and $\h(S')\leq \D(G)-1$. Since $\D(G)-1\geq \dd^*(G)$ (from \eqref{trivial-dav-bound}), we can apply Theorem \ref{CCD-d*} to $S'|S$ with $n=\D(G)-1$. If Theorem \ref{CCD-d*}(i) holds with resulting subsequence $S''|S$ of length $|S''|=|S'|$, then (recall $|S''|-(\D(G)-1)=|G|$ and note that there is natural correspondence between $k$-sums and $(|S''|-k)$-sums; see also \cite{hamconj-paper}) $$G=\Sigma_{\D(G)-1}(S'')=\sigma(S'')-\Sigma_{|G|}(S'').$$ Thus $\Sigma_{|G|}(S'')=G$, and the proof is complete in view of Step 1. This also completes the base case when $|G|=p$ is prime. If instead Theorem \ref{CCD-d*}(ii) holds, then there exists a coset $\alpha'+K$, with $K<G$ proper and nontrivial, such that all but at most $|G/K|-2$ terms of $S$ are from $\alpha'+K$.  Now applying the induction hypothesis to the subsequence of $S$ consisting of all terms from $\alpha'+K$, which has length at least $$|S|-|G/K|+2\geq |G|+\D_A(G)+1-|G/K|\geq |K|+\D_A(G)-1,$$ completes the proof in view of Step 1. So we may assume $|A|\geq 2$.

\bigskip

\textbf{Step 3.} Next, we establish some basic properties for the subgroups $G[d]$ and $G_0[d]$ under the assumption $G_0[d]\cap (\gamma+G)$ is nonempty. In view of Remark A, the hypothesis of Step 3 holds whenever there exists some $A_i|\A$ with $|A_i|=1$.

In view of \eqref{gen-cond} and the assumption $G_0[d]\cap (\gamma+G)\neq \emptyset$, it follows that $G+G_0[d]=G_0$. Thus \be\label{numberbound}G_0/G_0[d]\cong G/G[d].\ee

Since $|A|>1$, it follows that $d<m$ and so $G_0[d]$ is a proper subgroup of $G_0$, and thus, in view of \eqref{numberbound}, $G[d]$ is a proper subgroup of $G$. Moreover, by choosing the representative $\gamma\in \gamma+G$ appropriately, we can assume \be\label{task}G_0[d]\cap (\gamma+G)=\gamma+G[d].\ee Consequently, if there are at least $\D_A(G)$ sets $A_i|\A$ with $|A_i|=1$, then  \eqref{mult-hyp} and Remark A ensure that $G[d]$ is nontrivial. Next note, since $G+G_0[d]=G_0$, that \be\label{dgequal}d\cdot G_0=d\cdot (G+G_0[d])=d\cdot G+d\cdot G_0[d]=d\cdot G.\ee Finally, if $A\cdot s_j=A_j|\A$ with $s_j\in G_0[d]$, then \eqref{task} implies that $A_j=A\cdot s_j\subseteq a_0\gamma+G[d]$.

\bigskip

\textbf{Step 4. } Now we setup the induction and establish the base case.
Let $\B|\A$ be
a sub-setpartition of length $|G|+\D_A(G)-1$.
We apply the Devos-Goddyn-Mohar Theorem to $-\delta+\B\in \mathscr{S}(G)$ using $n=|G|$ (where $\delta$ is as given in the hypotheses). Hence, letting $L=\mathsf{H}(\Sigma^\cup_{|G|}(-\delta+\B))=\mathsf{H}(\Sigma^\cup_{|G|}(\B))$, we have \be\label{DGM-bound-for-us}|\Sigma^\cup_{|G|}(\mathscr{B})|\geq (\Summ{g\in G_0/L}\min\{|G|,\,\vp_g(\mathsf{S}(\phi_L(\B))\}-|G|+1)|L|.\ee If $L=G$, then the proof is complete (taking $H=G$ and $\alpha=\gamma$) in view of Step 1 and the definition of $\B$. Therefore we may assume $L<G$ is proper.

We proceed by induction on $|G|$.  The base case is when $|G|=p$ is prime, and thus $L<G$ must be trivial. Thus  from Step 3 it follows that (recall Step 3 implies $G[d]<G$ is proper) $$|G_0[d]\cap (\gamma+G)|\leq 1.$$ As a result, \eqref{mult-hyp} and Remark A imply that at least $|G|-1$ sets $A_i|\B$ have cardinality greater than one, whence \eqref{DGM-bound-for-us} yields $|\Sigma^\cup_{|G|}(\mathscr{B})|=|G|$, contradicting that $L<G$ is proper.

\bigskip

\textbf{Step 5.} Next, we handle the case when $L$ is nontrivial. If \eqref{mult-hyp2} holds for $\phi_L(\B)$, then applying the induction hypothesis to $\phi_L(\B)$ (with $n=|G|$), we contradict that $L$ is the maximal period for $\Sigma^\cup_{|G|}(\B)$ (in view of \eqref{goal-big}). Therefore we instead conclude that there exists $\delta'\in \delta+G$ such all but at most $|G/L|-2$ sets from $\B$ are contained in $\delta'+L$. Thus, letting $\C|\B$ be the sub-setpartition consisting of all sets contained in $\delta'+L$, we see that \be\label{ruth}|\C|\geq |\B|-(|G/L|-2)\geq |G|+\D_A(G)+1-|G/L|\geq |L|+\D_A(G)-1.\ee
Hence, in view of Remark B, we see that we can apply the induction hypothesis to $\C$, which  completes the proof in view of Step 1. So we now assume $|L|=1$.

\bigskip

\textbf{Step 6: } We proceed to conclude the proof by finishing the case when $L$ is trivial.  Let $\A_1|\A$ be the sub-setpartition consisting of all sets of size one. If $|\A_1|\leq \D_A(G)$, then there are at least $|G|-1$ sets $A_i|\B$ with $|A_i|\geq 2$, whence \eqref{DGM-bound-for-us} implies $|\Sigma^\cup_{|G|}(\B)|\geq |G|$, contradicting that $L<G$ is proper. Therefore we may assume $$|\A_1|\geq \D_A(G)+1,$$
whence Step 3 implies $G[d]\leq G$ is a proper and nontrivial subgroup. Consequently, $d\cdot G\leq G$ is also a proper, nontrivial subgroup (recall $d\cdot G\cong G/G[d]$).
Factor $\A=\A_1\A_2$, so $\A_2|\A$ is the sub-setpartition consisting of all sets $A_j$ of size at least two.

Suppose $|\A_1|\geq |G[d]|+\D_A(G)-1.$ Then, since  $G[d]<G$ is proper, it follows, in view of  Step 3 and Remark A, that we can apply the induction hypothesis to $\A_1$, which completes the proof in view of Step 1. So we may assume $$|\A_1|\leq |G[d]|+\D_A(G)-2.$$

Consequently, \be\label{yoda}|\A_2|\geq |S|-|G[d]|-\D_A(G)+2\geq |G|-|G[d]|+1\geq |d\cdot G|.\ee In view of \eqref{dgequal}, we have \be\label{dk}A_i=A\cdot s_i\subseteq a_0s_i+d\cdot G_0=a_0s_i+d\cdot G\ee  for all $i$. Thus, in view of \eqref{yoda} and $d\cdot G\leq G$ nontrivial, we see that we can apply Lemma \ref{keylemma} to $\A_2$ (with subgroups $d\cdot G\leq G_0$). Let $\C|\A_2$ be the resulting sub-setpartition (corresponding to $\A''$ in the lemma) and $H\leq d\cdot G<G$ the resulting nontrivial subgroup. Factor $\A=\C\Dc$.

Noting that \be\label{kiss}|\Dc|=|S|-|H|+1\geq |G|+\D_A(G)-|H|\geq |G/H|+\D_A(G)-1,\ee we see that we can apply the induction hypothesis to $\phi_H(\Dc)$. Let $K/H\leq G/H$, $\phi_H(\alpha)$, $\phi_H(\beta)$ and $\phi_H(S'),\,\phi_H(S'')\in \Fc(\phi_H(\alpha+K))$ be as given by the induction hypothesis, where $S'|S$, $S''|S'$, $\alpha\in \gamma+G$ and $\beta\in \delta+G$. Let $\B'=(A\cdot s_{j_1})\cdots (A\cdot s_{j_{|S'|}})$, where $S'=s_{j_1}\cdots s_{j_{|S'|}}$, and likewise define the setpartition $\B''$ using $S''$.

Suppose $K\leq G$ is a proper subgroup. Then, noting that $A_j\subseteq \beta+K$ for $A_j|\B'$ and that \eqref{smallsubsequence}, \eqref{kiss} and $H\leq K$ imply \ber\nn|\B'|&\geq& |\Dc|-|G/K|+1\geq|G|+\D_A(G)-|H|-|G/K|+1\\\nn&\geq & |G|+\D_A(G)-|K|-|G/K|+1\geq |K|+\D_A(G)-1,\eer we see in view of Remark B that we can apply the induction hypothesis to $\B'$, which completes the proof in view of Step 1. So we may assume $K=G$.

Since $K=G$, it follows, from \eqref{goal-small} holding for $\phi_H(\B'')$ with subgroup $K/H=G/H$,  and from \eqref{tryst} holding with subgroup $H$ for $\C$, that $$|\Sigma_{|H|-1+|G/H|}^\cup(\C\B'')|=|G|.$$ Thus, extending $\C\B''|\A$ to a sub-setpartition $\C'|\A$ with $|\C'|=|G|+\D_A(G)-1$, which is possible since $$|\C\B''|=|H|+|G/H|+\D_A(G/H)-2\leq |G|+\D_A(G)-1,$$ we see, in view of  \ber\nn &&|H|+|G/H|-1\leq |G|\\&&\nn |\C\B''|+(|G|-(|H|+|G/H|-1))=\\&&|H|+|G/H|+\D_A(G/H)-2+|G|-|H|-|G/H|+1\leq |G|+\D_A(G)-1,\nn\eer that $$|\Sigma_{|G|}^\cup(\C')|=|G|.$$ Consequently, in view of Step 1, the theorem holds with subgroup $H=G$, $\alpha=\gamma$ and $\beta=\delta$, completing the final step of the proof.
\end{proof}

\end{document}